\documentclass[12pt]{article}
\usepackage{latexsym,color,amsmath,amsthm,amssymb,amscd,amsfonts}
\usepackage{pgfplots}
\usetikzlibrary{positioning}
\usepackage{tikz}
\usetikzlibrary{arrows}
\usepackage{scalefnt}
\usepackage[font=small,labelfont=bf]{caption}
\usepackage{float}
\usepackage{graphicx}
\usepackage{amsopn}
\usepackage{relsize}
\usepackage{mathrsfs}
\usepackage{upgreek}
\usepackage[pdftex]{hyperref}

\renewcommand{\epsilon}{\varepsilon}
\newtheoremstyle{nonum}{}{}{\itshape}{}{\bfseries}{.}{ }{\thmnote{#3}}

\newtheorem{thm}{Theorem}[section]
\newtheorem*{thm*}{Theorem}

\newtheorem{lem}[thm]{Lemma}
\newtheorem{prop}[thm]{Proposition}
\newtheorem{rem}[thm]{Remark}

\newtheorem{definition}[thm]{Definition}
\newtheorem*{definition*}{Definition}

\newtheorem*{rems*}{Remarks}

\theoremstyle{nonum}

\newcommand{\R}{\mathbb R}
\newcommand{\E}{\mathbb E}

\def\A{{\cal A}}
\def\L{{\cal L}}

\def\J{{\cal J}}

\newcommand{\iprod}[2]{\langle #1,#2 \rangle} 

\newcommand{\kn}{\mathcal{K}^n}			

\def\calL{{\cal L}}
\def\L{{\calL}}

\def\vol{{\rm Vol}}

\def\Jf{\J\varphi}

\def\epi{{\rm epi}\,}
\def\epif{\epi \varphi }
\def\epiff{\epi(\varphi)}
\def\epij{\epi(\J\varphi)}
\def\gcf{{\rm{Cvx}}_0(\R^n)}
\def\cvx{\rm{Cvx}}

\newcommand{\jinfp}{\diamond\kern-0.62em{\cdot}} 

\begin{document}
\title{A Santal\'{o}-type Inequality for the $\J$ Transform}
\date{}
\author{D.I. Florentin, A. Segal}
\maketitle
\begin{abstract}
This paper deals with an analog of the Mahler volume product related
to the $\J$ transform acting in the class of geometric convex
functions $\gcf$. We provide asymptotically sharp bounds for
the quantity $s^{\J}(f) = \frac{\int e^{-\J f}}{\int e^{-f}}$ and
characterize all the extremal functions.
\end{abstract}

\section{Introduction}
The classical Blaschke-Santal\'{o} inequality states that the Mahler
volume, an af{}fine invariant functional given by
\[
s(K) =
\inf
\left\{
\vol_n(K - x) \cdot \vol_n((K - x)^\circ) \,\Big|\, x\in K
\right\},
\]
is maximized by ellipsoids. Here $K^\circ$ stands
for the dual body of $K$, namely
\[
K^\circ =
\left\{
y\in\R^n \,\Big|\, \forall x\in K,\iprod{x}{y}\le 1
\right\}.
\]
Analysis of the Mahler volume dates back to 1917, when Blaschke
proved that in dimensions $n = 2,3$, every convex body $K$ satisfies
$s(K) \le s(B^n_2)$ (see \cite{Bla}). Santal\'{o} \cite{Sant}
extended this
inequality to every dimension (see \cite{MeyPaj} for a simple proof
due to Meyer and Pajor). However, finding the minimal value of the
Mahler volume remains an open problem to this day. Mahler conjectured
that for centrally symmetric bodies, the Mahler volume attains its
minimum value on cubes, i.e. $s(K)\ge s(B_\infty^n) = \frac{4^n}{n!}$
(it is now known that any Hanner polytope has the same Mahler volume
as the cube). The Mahler conjecture would imply that the minimum and
maximum of $s(K)$ differ only by a factor of $c^n$, a fact which was
verified in the celebrated Bourgain-Milman theorem \cite{BM}. They
proved that there exists a universal constant $c$ such that:
\begin{equation*}\label{Ineq-Bour-Mil}
c \le \left( \frac{s(K)}{s(B_2^n)} \right) ^\frac1n,
\end{equation*}
thus settling the Mahler conjecture in the asymptotic sense. In
recent years, several new proofs of the Bourgain-Milman inequality
have been discovered, by Kuperberg \cite{Kup}, Nazarov \cite{Naz},
and by Giannopoulos, Paouris, and Vritsiou \cite{GPV}, Kuperberg
setting the value of the largest known constant $c$ at
$\frac{\pi}4$. In the general case (where the body $K$ is not assumed
to be centrally symmetric), minimizers are conjectured to be centered
simplices. Mahler \cite{Mahler} solved the planar case, and the
(centrally symmetric) three dimensional case has been settled
recently by Iriyeh and Shibata \cite{IriShi} (see also \cite{FHMRZ}
for a simplified proof).

Similar questions were considered for the class of convex functions
in $\R^n$, which we denote by $\cvx(\R^n)$. The analog of volume of a
convex function $\varphi$ is usually defined to be the integral
$\int_{\R^n}e^{-\varphi}$, and the dual of $\varphi$ is obtained by
applying the Legendre transform
\[
(\mathcal{L}\varphi)(y) = \sup_{x\in\R^n} (\iprod{x}{y}-\varphi(x)).
\]
The following infimum over translations is a functional analogue of
the Mahler volume.
\[
s^{\L}(\varphi) = \inf_{a \in \R^n}
\left\{\int e^{-T_a\varphi}\int e^{-\L T_a \varphi}\right\}.
\]
Here (and only here) $T_a\varphi(x) = \varphi(x-a)$. Ball proved in
\cite{Ball2} for even functions, and Artstein, Klartag, and Milman
proved in \cite{AKM} for any function, that
\begin{equation}\label{eq-ML-upper}
s^{\L}(\varphi) \le s^{\L}\left( \frac{|\cdot|^2}{2} \right) =
(2\pi)^n.
\end{equation}
The lower bound for $s^\L$ was established by Klartag and Milman in
\cite{KM} for even functions, and by Fradelizi and Meyer in
\cite{FradMey} for all functions, namely
\begin{equation}\label{eq-ML-lower}
c^n \le s^{\L}(\varphi),
\end{equation} 
where $c>0$ is some universal constant. For a large family of convex
functions there is another choice of duality besides the Legendre
transform, namely the polarity transform $\A$, which first appeared
in \cite{Rockafellar}. Artstein and Milman proved in \cite{AM-Hidden}
that in $\cvx_0(\R^n)$, the class of non-negative convex functions
vanishing at $0$, the only dualities (i.e. order reversing
involutions) are the Legendre transform $\L$ and the polarity
transform $\A$, which may be defined by
\[
(\A\varphi)(x) = \sup_y \frac{\iprod{x}{y}-1}{\varphi(y)}.
\]
One may similarly consider the Mahler volume with respect to $\A$:
\[
s^{\A}(\varphi) = \int e^{-\varphi}\int e^{-\A \varphi}.
\]
It was shown in \cite{ArtSlom} for even functions, that
\begin{equation}\label{eq-MA-upper-lower}
c^n \left(\vol_n\left(B_2^n\right)\right)^2
\le s^{\A}(\varphi) \le
\left(\vol_n\left(B_2^n\right) n! \right)^2 \left(1+\frac{C}{n}\right).
\end{equation}
Finding maximizers and minimizers of $s^{\A}$ (and even proving their
existence) remains an open problem. The composition $\J = \A \L =
\L \A\,$ of the two order reversing involutions $\L$ and $\A$ is an
order preserving involution (see \cite{AM-Hidden} and \cite{AFM} for
more details on the $\J$ transform). Since $\J$ is order preserving,
the product $\int e^{-\varphi} \int e^{-\J \varphi}$ cannot be
bounded from above, or bounded away from zero. Thus, it makes sense
to consider the following quantity, whenever $\int e^{-\varphi}$ and
$\int e^{-\J \varphi}$ are both positive and finite:
\[
s^{\J}(\varphi) = \frac{\int e^{-\J \varphi}}{\int e^{-\varphi}}
\]
To avoid trivial exceptions we use the convention $\frac00 =
\frac\infty\infty = 1$, thus defining $s^\J$ on the whole of
$\cvx_0(\R^n)$ (as it is not hard to verify that $\int e^{-\varphi} =
0$ if and only if $\int e^{-\J\varphi} = 0$, and similarly
$\int e^{-\varphi} = \infty$ if and only if $\int e^{-\J \varphi} =
\infty$).
The purpose of this note is the study of the functional $s^{\J}$. In
our first theorem we describe all maximizers and minimizers of
$s^{\J}$. To this end we use the following notation. The class of all
compact convex sets in $\R^n$, with non empty interior, which contain
the origin is denoted by $\kn$. For any $K\in\kn$ and $a > 0$, let
\[
\psi_{K, a} = \max\left\{1_K^\infty,\, a||\cdot||_{K}\right\},
\]
where $||\cdot||_{K}$ is the Minkowski functional of the body $K$,
and the convex indicator function $1_K^\infty$ is defined to be $0$
on $K$ and $+\infty$ otherwise. It turns out
(see e.g. the proof of Theorem \ref{Thm-maximizers})
that the value of $s^\J(\psi_{K, a})$
depends only on $a$, and not on the body $K$. Thus in the following
theorem, which describes all the maximizers of $s^\J$, the Euclidean
ball $B_2^n$ may be replaced by any $K\in\kn$.
\begin{thm}\label{Thm-maximizers}
For every $n \ge 1$ there exists $a_n > 0$ such that for every
$\varphi \in \cvx_0(\R^n)$, we have
\[
s^{\J}(\J\psi_{B_2^n,a_n}) \le
s^{\J}(\varphi)      \le
s^{\J}(\psi_{B_2^n, a_n}).
\]
Equality holds only if $\,\varphi = \psi_{K, a}$ for some
$K\in\kn, a > 0$. Moreover, $\lim_{n\to\infty} n a_n = 1$.
\end{thm}
Note that $\J\psi_{K, a} = \psi_{\frac{1}{a}K, \frac{1}{a}}$ (see e.g.
\cite{AM-Hidden}). We will see in Lemma \ref{lem-Scaling} below, that
$s^\J$ attains the same value on $\psi_{\frac{1}{a}K, \frac{1}{a}}$
and on $\psi_{K, \frac{1}{a}}$, thus Theorem \ref{Thm-maximizers}
implies that
\[
\psi_{K, a}\mbox{ is a maximizer}
\iff
\psi_{K, \frac{1}{a}} \mbox{ is a minimizer.}
\]
Our second theorem provides asymptotically sharp bounds on the
extremal values of $s^\J$.
\begin{thm}\label{Thm-asymp}
There exist universal constants $c, C > 0$ such that for $n$ large
enough:
\[
\left( 1 + \frac{c}{n} \right)n!
\le \max_{\varphi\in \gcf} \left\{s^{\J}(\varphi)\right\} \le
\left( 1 + \frac{C}{n} \right)n!.
\]
\end{thm}

%
%
%
%
%

\noindent {\bf Acknowledgements:}
The authors would like to thank Fedor Nazarov and Shiri
Artstein-Avidan for useful discussions. The first named author was
partially supported by the U.S. National Science Foundation Grant
DMS-1101636, and also partially supported by the AMS-Simons Travel
Grant, which is administered by the American Mathematical Society
with support from the Simons Foundation.

\section{Preliminaries - basic properties of 
$s^{\J}$}
In this section we analyze the action of $\J$ using properties of the
point map which induces it, and prove that $\J$ commutes with the
action of symmetrizing a function. This allows us to apply the known
bounds for $s^\L$ and $s^\A$ and show that $s^\J$ is bounded (with a
non optimal constant). More importantly, it sets the
groundwork to reducing the problem of finding optimal bounds  for
$s^\J$ (and extremal functions), to a certain two dimensional
optimization problem.
%
We begin by considering the function
\mbox{$F:\R^n\times \R^+ \to \R^n \times \R^+$} given by
\[
F(x,z) = \left( \frac{x}{z}, \frac{1}{z} \right).
\]
The map $F$ is an involution that induces the $\J$ transform (see
\cite{AM-Hidden}), in the following sense. Let $\epi(\phi) =
\left\{ (x, z)\in \R^n\times\R^+ \,:\, \phi(x) < z \right\}$ denote
the epi-graph of a function $\phi$. Then for any $\varphi \in \gcf$
we have:
\[
\epij=F(\epif),
\]
For any $z\ge 0$, let $H_z=\left\{
(x,z)\,|\, x\in\R^n\right\} \subset\R^{n+1}$ denote the hyperplane at
height $z$, and let $L_z(\varphi)=\left\{ x\in\R^n\,|\, (x,z)\in
\epiff\right\} \subseteq \R^n$ denote the corresponding slice of
$\epiff$. Note that the map $F$, restricted to $H_z$, is simply a
dilation by $\frac{1}{z}$. Consequently, we have the following simple
relation between level sets of $\varphi$ and $\Jf$:
%
\begin{equation}\label{Eq-J-dilation}
L_{\frac{1}{z}}(\Jf) = \frac{1}{z} L_z(\varphi).
\end{equation}
For a unit vector $u\in S^n$ which is perpendicular to the $z$ axis,
we define the operator $S_u:\gcf \to \gcf$ by setting
$\epi (S_u\varphi)$ to be the Steiner symmetrization of $\epiff$ with
respect to the direction $u$ (note that indeed $S_u\varphi\in \gcf$).
The following proposition is a direct consequence of
\eqref{Eq-J-dilation}.
\begin{prop}\label{prop-JSu=SuJ}
$S_u$ commutes with $\J$. That is, 
\begin{equation}\label{Eq-Steiner-commutes}
S_u (\Jf) = \J (S_u \varphi).
\end{equation}
\end{prop}
Consider the measures $\mu, \nu$ on $\R^n\times\R^+$ with densities
$e^{-z}$, $e^{-\frac{1}{z}}z^{-(n+2)}$ respectively.
It was shown in \cite{AFS} that for any $\varphi\in \gcf$, 
\begin{equation}\label{Eq-integrals-by-munu}
\mu(\epif) = \int_{\R^n}e^{-\varphi},\qquad
\nu(\epif) = \int_{\R^n}e^{-\Jf}.
\end{equation}
Since the densities of $\mu$ and $\nu$ only depend on $z$ (and not on
$x$), and the volumes of 
$L_z(S_u\varphi)$ and $L_z(\varphi)$ are equal, we have
\begin{equation}\label{Eq-Steiner-preserves-volumes}
\int_{\R^n}e^{-\varphi} = \int_{\R^n}e^{-S_u\varphi},\qquad
\int_{\R^n}e^{-\Jf} = \int_{\R^n}e^{-\J S_u\varphi}.
\end{equation}
Applying $n$ successive Steiner symmetrizations in directions
$e_1,\dots,e_n$ results in an unconditional function, i.e. a function
satisfying $\eta(\pm x_1,\dots,\pm x_n) = \eta(x_1,\dots,x_n)$. In
particular $\eta$ is even, so we can easily conclude that the ratio
$s^{\J}$ is bounded above and below. Recall that $\J = \L \circ \A$,
so we may write:
\[
s^\J(\eta) =
\frac{\int e^{-\J \eta}}{\int e^{-\eta}}
\frac{\int e^{-\A \eta}}{\int e^{-\A \eta}} = 
\frac{\int e^{-\L (\A \eta)} \int e^{-\A \eta}}{\int e^{-\A \eta}\int e^{-\eta}} =
\frac{s^\L(\A \eta)}{s^\A(\eta)} =
\frac{s^\A(\L \eta)}{s^\L(\eta)} .
\]
This implies
\[
\max\left\{
\frac{\min\{s^\L\}}{\max\{s^\A\}},\,
\frac{\min\{s^\A\}}{\max\{s^\L\}}
\right\}\le
s^\J(\eta)\le
\min\left\{
\frac{\max\{s^\L\}}{\min\{s^\A\}},\,
\frac{\max\{s^\A\}}{\min\{s^\L\}}
\right\}.
\]
Using the bounds \eqref{eq-ML-upper}, \eqref{eq-ML-lower}, and
\eqref{eq-MA-upper-lower} yields
\[
s^\J(\eta)\le
C^n n!
\]
for some $C>1$. However, these bounds are far from optimal. In fact,
we shall see that the maximal value of $s^\J$ is not much larger than
$n! = s^\J(1_K^\infty)$. To summarize, we have shown the following.
\begin{rem}\label{rem-MJ-bounded}
Denoting
\[
\lambda_n = \sup_{\varphi\in \cvx_0(\R^n)} s^\J (\varphi),
\]
we have
\begin{equation}\label{eq-MJ-prelim-bounds}
n! \le \lambda_n \le C^n n!
\end{equation}
\end{rem}
%
Another important property of the $\J$ transform is that its volume
is preserved under rearrangement. For a convex function
$\varphi \in \gcf$ we define the symmetric rearrangement of
$\varphi$ to be the function $\varphi^\ast:\R^n\to\R^+$ satisfying
that $L_z(\varphi^\ast)$ is the Schwartz symmetrization of
$L_z(\varphi)$, i.e. a centered Euclidean ball with the same volume.
\begin{eqnarray*}
	\int_{\R^n}e^{-\varphi}
	&=&
	\int_0^\infty \vol_n \left(x: e^{-\varphi}\ge t\right) dt =
	\int_{-\infty}^\infty \vol_n \left(x : e^{-\varphi} \ge e^{-z}\right) e^{-z}dz \\ \\ 
	&=&
	\int_0^\infty \vol_n \left(x : \varphi \le z\right) e^{-z} dz =
	\int_0^\infty \vol_n\left(L_z(\varphi)\right)  e^{-z}dz.
\end{eqnarray*}
A similar formula holds for the volume of $\J \varphi$, as a weighted
integral of the $n$-dimensional volumes of its level sets.
\begin{lem} \label{lem-Steiner-J}
Let $\varphi \in \gcf$. Then
\[
\int_{\R^n} e^{-\J \varphi} =
\int_0^\infty \vol_n\left(L_z(\varphi)\right)e^{-\frac1z}z^{-(n+2)}dz.
\]
\end{lem}
\begin{proof}
We use \eqref{Eq-integrals-by-munu} to get
\begin{eqnarray*}
\int_{\R^n} e^{-\J \varphi}
&=&
\int_{\epi(\varphi)}e^{-\frac1z}z^{-(n+2)}dzdx =
\int_0^\infty \vol_n\left(L_z(\varphi)\right) e^{-\frac1z}z^{-(n+2)}dz.
\end{eqnarray*}
\end{proof}
From Lemma \ref{lem-Steiner-J} we conclude that the functional
$s^{\J}(\varphi)$ depends only on the volumes of level sets of
$\varphi$. Thus, by replacing the level sets of $\varphi$ with balls,
we conclude that it is enough to maximize $s^{\J}$ over spherically
symmetric functions. Given a spherically symmetric function
$\varphi:\R^n \to \R^+$, let $\psi = \varphi|_{\R^+}$ denote its
restriction to a ray. Since $\varphi(x)$ depends only on $|x|$ we
have:
\[
\vol(\J \varphi) =
\int e^{-\J \varphi} =
\int_{\R^n}dx\int_{\varphi(x)}^{\infty}e^{-\frac1z}z^{-(n+2)}dz =
n\kappa_n \int_0^\infty r^{n-1}dr\int_{\psi(r)}^{\infty}e^{-\frac1z}z^{-(n+2)}dz.
\]
Thus we define the $2$-dimensional measure $\nu_2$ on the first quadrant by
\[
d\nu_2=n\kappa_n r^{n-1}e^{-\frac1z}z^{-(n+2)}dzdr.
\]
We have seen that $\vol(\J \varphi) = \nu(\epi(\varphi)) =
\nu_2(\epi(\psi))$.
Similarly, defining
\[
d\mu_2=n\kappa_n r^{n-1}e^{-z}dzdr,
\]
yields $\vol(\varphi) = \mu(\epi(\varphi)) = \mu_2(\epi(\psi)$.
Therefore we define
\begin{equation}\label{eq-Def-Mn}
s^\J_n(\psi) =
\frac{\nu_2(\epi(\psi))}{\mu_2(\epi(\psi))} = s^\J(\varphi),
\end{equation}
and conclude that 
\begin{equation}\label{Eq-2D-sup}
\lambda_n = 
\sup_{\varphi\in\cvx_0(\R^n)} s^{\J}(\varphi) \,=
\sup_{\psi\in\cvx_0(\R^+)} s^\J_n(\psi).
\end{equation}
Another useful property of the ratio $s^{\J}$ is its invariance under
rescaling. Namely, defining $\varphi_a\in \cvx_0(\R^n)\,$ by
$\,\varphi_a(x) := \varphi(ax)$, we have:
\begin{lem} \label{lem-Scaling}
Let $\varphi \in \gcf$ and $a>0$. Then, $s^{\J}(\varphi_a) = s^{\J}(\varphi)$.
\end{lem}
\begin{proof}
Since $L_z(\varphi_a) = \frac1a L_z(\varphi)$, we have
\[
\vol_n(L_z(\varphi_a)) = \frac{1}{a^n} \vol_n(L_z(\varphi)).
\]
By Lemma \ref{lem-Steiner-J}, we get
$\vol\left(\J \varphi_a\right) =
\frac{1}{a^n} \vol\left(\J \varphi\right)$. Similarly,
$\vol\left(\varphi_a\right) =
\frac{1}{a^n} \vol\left(\varphi\right)$, and the proof follows.
\end{proof}

\section{Characterization of maximizers of $s^\J$}
We know by \eqref{eq-MJ-prelim-bounds} that $s^{\J}$ is bounded from
above and below. In this section we prove Theorem \ref{Thm-maximizers}
by showing that $s^\J$ attains its maximum (which implies that its
minimum is also attained). The maximizer is found
by reducing the problem to a two dimensional optimization problem,
then using the scaling invariance established in Lemma
\ref{lem-Scaling}, thus restricting to the class
$\cvx_{0,z}(\R^+) \subset \cvx_0(\R^+)$ defined as follows, for
$z > 0$.
\[
\cvx_{0,z}(\R^+) =
\left\{\varphi\in \cvx_0(\R^+)\,:\, (1,z)\in \partial \epi(\varphi) \right\},
\]
or equivalently, $\varphi\in\cvx_{0,z}(\R^+) \iff L_z(\varphi) = [0,1]$. By
Lemma \ref{lem-Scaling} we have
\begin{equation}\label{Eq-sup-on-cvxz}
\lambda_n \quad=
\sup_{\varphi\in \cvx_0(\R^+)}s^\J_n(\varphi)\quad =
\sup_{\varphi\in\cvx_{0,z}(\R^+)} s^\J_n(\varphi).
\end{equation}
Moreover, any $\varphi\in\cvx_{0,z}(\R^+)$ satisfies
\[
\hat{\min}\left\{ 1^\infty_{[0,1]}, l_z \right\} \le
\varphi \le
\max\left\{1^\infty_{[0,1]}, l_z\right\},
\]
where $l_z$ is a linear function with slope $z$, and
$\hat{\min}(\eta, \xi)$ is defined to be the largest convex
function smaller than $\min(\eta, \xi)$. This implies the existence
of positive constants $c=c(n,z)$, and $C=C(n,z)$ such that
\begin{equation}\label{Eq-bounds-on-cvxz}
c(n,z) \le
\mu_2(\epi(\varphi)) \le
C(n,z).
\end{equation}
We define a signed measure $\Delta$ on $\R^+\times \R^+$ by
\begin{equation}
\Delta = \nu_2 -\lambda_n \mu_2.
\end{equation}
We get, for any $\varphi\in \cvx_{0}(\R^+)$ that
\begin{equation}\label{eq-Delta-negative}
\Delta(\epi(\varphi)) =
\nu_2(\epi(\varphi)) - \lambda_n \mu_2(\epi(\varphi)) \le
0.
\end{equation}
\begin{rem}\label{rem-same-maximizers}
Clearly, $\Delta(\epi(\varphi)) = 0$ if and only if
$s^\J_n(\varphi) = \lambda_n$, ~i.e. if and only if
$\varphi\in\cvx_0(\R^+)$ is a maximizer of $s^\J_n$.
\end{rem}
The density of the signed measure $\Delta$ is given by
\[
d\Delta =
n\kappa_nr^{n-1}\left(e^{-\frac1z}z^{-(n+2)}-\lambda_n e^{-z}\right)dzdr =
n\kappa_nr^{n-1}m(z)dzdr,
\]
where $m(z) = e^{-\frac1z}z^{-(n+2)}-\lambda_n e^{-z}$.
\begin{lem}\label{lem-3-sign-changes}
The function $m:\R^+\to\R$ changes sign at three points
$z_1 < z_2 < z_3$.
\end{lem}
\begin{proof}
Consider the function
$f(z) =
e^{\frac{1}{n+2}\left(z-\frac{1}{z}\right)}
- \lambda_n^{\frac{1}{n+2}}z$. Since
%
\[
m(z)=0 \iff
e^{-\frac1z}z^{-(n+2)} = \lambda_n e^{-z} \iff
e^{\frac{1}{n+2}\left(z-\frac{1}{z}\right)} =
\lambda_n^{\frac{1}{n+2}}z,
\]
we conclude that $m$ and $f$ have the same roots. Since
\[
f''(z) =
\frac{e^{\frac{\left(z-\frac{1}{z}\right)}{n+2}}}{z^4(n+2)^2}\cdot
\left(
\left(z^2 + 1\right)^2
-2(n+2)z
\right),
\]
we get
\[
f''(z)=0 \iff z^4 + 2z^2 + 1 = 2(n+2)z.
\]
Since $z\mapsto z^4+2z^2+1$ is strictly convex and $z\mapsto 2(n+2)z$
is linear, $f''(z)$ has at most two roots, which implies that $f$ has
at most four roots. Moreover, $m(0^+) = -\lambda_n$ is negative and
$m(\infty) = 0^+$, which implies that one of the two following holds.
\begin{enumerate}
	\item There exists $z_0 > 0$ such that $m \le 0$ on $[0, z_0]$ and
	$0\le m$ on $[z_0, \infty)$.
	\item There exist $0 < z_1 < z_2 < z_3$ such that $m \le 0$ on
	$[0,z_1]\cup[z_2,z_3]$ and $0\le m$ on
	$[z_1, z_2] \cup [z_3,\infty)$.
\end{enumerate}
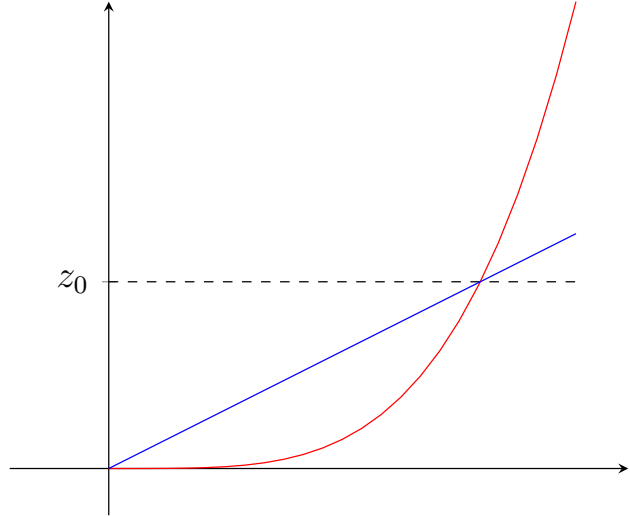
\begin{figure}[h]
	\centering
	\pgfplotsset{my style/.append style={axis x line=middle, axis y line=
			middle, axis equal }}
	\begin{tikzpicture}[scale=1.2]
	\begin{axis}[my style, xmin=-1, ymin = -1, xmax = 10, ymax = 10, xtick={0}, xticklabels={0}, ytick={4}, yticklabels={$z_0$}]
	\addplot [color=black, domain=0:10, dashed] {4};
	\addplot [color=red, domain=0:10] {0.001*x^4};
	\addplot [color=blue, domain=0:10] {4/7.952*x};
	\end{axis}
	\end{tikzpicture}
	\caption{The case where $m$ changes sign only at $z_0$.} \label{fig:Figure1}
\end{figure}
Next, we shall exclude the first case. Consider the function
$l(x) = z_0 x$ defined on $\R^+$, and for any $\varphi \in
\cvx_{0, z_0}(\R^+)$, define the sets:
\[
X = \epi(\varphi)\cap\epi(l),\qquad
Y = \epi(\varphi)\setminus\epi(l),\qquad
Z = \epi(l)\setminus\epi(\varphi).
\]
By convexity, $Y\subset \left\{ (x,z)\, :\, z\in[0, z_0] \right\}$,
we have $\Delta(Y) < 0$. Similarly, $\Delta(Z) > 0$. Therefore
\begin{eqnarray*}
\Delta(\epi(\varphi)) &=&
\Delta(X) + \Delta (Y) <
\Delta(X) < 
\Delta(X) + \Delta (Z) = 
\Delta(\epi(l)) = \\ &=&
\nu \left(\epi\left(\|\cdot\|_{\frac{1}{z_0}B}\right)\right) - \lambda_n
\mu \left(\epi\left(\|\cdot\|_{\frac{1}{z_0}B}\right)\right) =
\frac{\kappa_n}{z_0^n} (1 - n!\lambda_n) < 0
\end{eqnarray*}
However, by \eqref{Eq-sup-on-cvxz}, for every $\epsilon > 0$ there exists
$\varphi \in \cvx_{0, z_0}(\R^+)$ such that
\[
\Delta(\epi(\varphi)) > -\epsilon\mu_2(\epi(\varphi)).
\]
Combining the above with \eqref{Eq-bounds-on-cvxz} we obtain
\[
0 >
\Delta(\epi(l)) >
\Delta(\epi(\varphi)) >
-\epsilon\mu_2(\epi(\varphi)) >
-\epsilon C(n, z_0),
\]
thus $\frac{\kappa_n}{z_0^n} (1 - n!\lambda_n) \in
(-\epsilon C(n, z_0) , 0)$ for any $\epsilon > 0$, which is a
contradiction. We are left with the second case, and the proof is
complete.
\end{proof}
The exclusion of the ``one root case'' in Lemma \ref{lem-3-sign-changes}
is based on improving (i.e. increasing) the measure $\Delta$ of an
epi-graph, by means of intersecting it with a ray (while relying on the
convexity of the epi-graph). We next extend this idea to improve the
measure $\Delta$ of an epi-graph, using the three roots of $m$. We use
it to show that a maximizer of $s^\J_n$ must be of the form
\[
T_{a, b, x_0}(x) =
\left\{
	\begin{array}{lr}
		   a      x                  &:  0\le x\le x_0\\
		(a + b)(x - x_0) + a x_0 &:~~~~~    x_0 \le x  \\
	\end{array}
\right\},
\]
where $a\in[0,\infty),\, b\in[0,\infty],\, x_0\in[0,\infty)$.
To this end we will define a mapping which assigns to each function
in $\cvx_0(\R^+)$ a function of the form $T_{a, b, x_0}$.
\begin{definition}
The map $T:\cvx_0(\R^+)\to \cvx_0(\R^+)$ is defined as follows.
Let $x_1, x_2, x_3$ be such that
$L_{z_i}(\varphi) = \left[0, x_i \right]$, where $z_1,\, z_2,\, z_3$
are the three points where $m$ changes sign. If $\varphi \equiv 0$,
set $T(\varphi) := \varphi$. Otherwise $x_1 \le x_2 \le x_3 < \infty$.
Set
%
%
%
%
\begin{eqnarray*}
a&=&\frac{z_1}{x_1} > 0,\qquad
b = \frac{z_3 - z_2}{x_3 - x_2} - a\ge 0,\\
L_1(x) &=& ax,\qquad \hskip 21pt
L_2(x) = (a+b)(x - x_2) + z_2.
\end{eqnarray*}
Set $T(\varphi) := \max\left\{ L_1, L_2 \right\}$ (see Figure
\ref{fig:Figure2}).

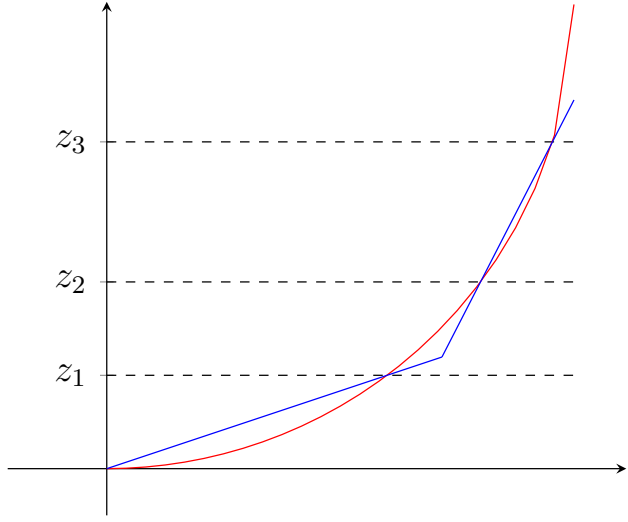
\begin{figure}[h]
	\centering
	\pgfplotsset{my style/.append style={axis x line=middle, axis y line=
			middle, axis equal }}
	\begin{tikzpicture}[scale=1.2]
	\begin{axis}[my style, xmin=-1, ymin = -1, xmax = 10, ymax = 10, ytick={2,4,7}, yticklabels={$z_1$, $z_2$, $z_3$}, xtick={0}, xticklabels={0}]
	\addplot [color=black, domain=0:10, dashed] {2};
	\addplot [color=black, domain=0:10, dashed] {4};
	\addplot [color=black, domain=0:10, dashed] {7};
	\addplot [color=red, domain=0:10] {10 - (100-x^2)^0.5};
	\addplot [color=blue, domain=0:7.17455] {0.3333*x};
	\addplot [color=blue, domain=7.1745:10] {1.9489*(x-8)+4};
	\end{axis}
	\end{tikzpicture}
	\caption{Definition of $T$: $\varphi$ in red, $T(\varphi)$ in blue.} \label{fig:Figure2}
\end{figure}
\end{definition}

\begin{rem}\label{rem-a-range}
Note that if the equation $L_1(x_0) = L_2(x_0)$ determines $x_0$
uniquely, then $T(\varphi) = T_{a, b, x_0}$, and $a x_0 \in
[z_1, z_2]$. If however, $L_1 = L_2$ has more than one solution,
then it is not hard to see that $\varphi$ is linear on $[0, x_3]$,
$b = 0$, and $T(\varphi) = L_1 = L_2$. Thus we may write in this case
too, that $T(\varphi) = T_{a, 0, x_0}$ and $a x_0 \in [z_1, z_2]$
(say, $x_0 := x_1$).
\end{rem}
The mapping $T$ improves upon the measure $\Delta$ of an epi-graph.
More precisely,
\begin{lem} \label{lem-Maximizers}
Let $\varphi \in \cvx_0(\R^+)$. Then,
$\Delta(\varphi) \le \Delta(T(\varphi))$. Moreover, if $\varphi$ is
not of the form $T_{a,b,x_0}$ then the inequality is strict.
\end{lem}
\begin{proof}
In Lemma \ref{lem-3-sign-changes} we have seen that there exist three
distinct points $z_1 < z_2 < z_3$ in which $m$ changes sign. This
implies that the sign of the density of the signed measure $\Delta$
is fixed on each of the following slabs:
\begin{eqnarray*}
	S_1 &=& \{(x,z): \,0  \le z \le z_1\} \\ 
	S_2 &=& \{(x,z): z_1 \le z \le z_2\} \\
	S_3 &=& \{(x,z): z_2 \le z \le z_3\} \\
	S_4 &=& \{(x,z): z_3 \le z        \}. 
\end{eqnarray*}	
By construction, we have
\begin{eqnarray*}
	\epi(T(\varphi)) \cap S_1 &\subseteq& \epi(\varphi) \cap S_1 \\
	\epi(T(\varphi)) \cap S_2 &\supseteq& \epi(\varphi) \cap S_2 \\
	\epi(T(\varphi)) \cap S_3 &\subseteq& \epi(\varphi) \cap S_3 \\
	\epi(T(\varphi)) \cap S_4 &\supseteq& \epi(\varphi) \cap S_4.
\end{eqnarray*}
Since the signed measure $\Delta$ is negative on $S_1,\,S_3$, and
positive on $S_2,\,S_4$, the statement follows. If one of the above
inclusions is strict, then (since the sets are convex) 
$\Delta(\varphi) < \Delta(T(\varphi))$. 
\end{proof}

\begin{lem}\label{lem-Maximizer-unique}
If $\varphi$ is not of the form $T_{a, b, x_0}$, then $\varphi$ is
not a maximizer of $s^\J_n$.
\end{lem}
\begin{proof}
If $\varphi$ is a maximizer of $s^{\J}_n$ then by Remark
\ref{rem-same-maximizers}, $\Delta(\varphi) = 0$. If $\varphi$ is not
of the form $T_{a,b,x_0}$, then by Lemma \ref{lem-Maximizers} we have
\[
0 = \Delta(\varphi) < \Delta(T(\varphi)),
\]
which is a contradiction to \eqref{eq-Delta-negative}.
\end{proof}
We next show that the supremum of $s^{\J}_n$ over functions of the
form $T_{a,b,x_0}$ is not (strictly) smaller than the supremum over all
of $\cvx_0(\R^+)$.

\begin{lem}\label{lem-same-sups}
For every $n\ge 1$ :
\[
\sup_{a,b,x_0} s^{\J}_n(T_{a,b,x_0}) =
\sup_{\varphi\in\cvx_0(\R^+)} s^{\J}_n(\varphi).
\]
\end{lem}
\begin{proof}
Denote
\[
\delta_n = \sup s^{\J}_n(T_{a,b,x_0}).
\]
Assume that $\delta_n < \lambda_n$. Recall that by
\eqref{Eq-bounds-on-cvxz}, for any $\psi\in\cvx_0(\R^+)$,
\begin{equation}\label{eq-more-bounds-on-cvxz}
\hat{\min}\left\{1_{[0,1]}^\infty, l_{z_1}\right\}
\le \psi \le
\max      \left\{1_{[0,1]}^\infty, l_{z_3}\right\}
\quad \Rightarrow \quad
0 < c(n,z_3) < \mu_2(\psi) < C(n,z_1)
\end{equation}
Choose $\epsilon =
(\lambda_n-\delta_n)\frac{c(n,z_3)}{2C(n,z_1)} > 0$, and let
$\tilde{\varphi}\in\cvx_0(\R^+)$ be such that $s^\J_n(\tilde{\varphi}) >
\lambda_n - \epsilon$. If $T_{\tilde{a},\tilde{b},x_0} = T(\tilde{\varphi})$,
then for $\varphi = \tilde{\varphi}_{x_0}$ one has
\[
T(\varphi) =
T(\tilde{\varphi}_{x_0}) =
\left(T(\tilde{\varphi})\right)_{x_o} =
T_{x_0 \tilde{a}, x_0 \tilde{b}, 1} \equiv
T_{           a ,            b , 1},
\]
and $s^\J_n(\varphi) = s^\J_n(\tilde{\varphi}) > \lambda_n - \epsilon$.
Moreover,
\begin{eqnarray*}
&&\hat{\min}\left\{1_{[0,1]}^\infty, l_{z_1}\right\}
\le \,\,\,\,\,\varphi\,\,\,\, \le
\max      \left\{1_{[0,1]}^\infty, l_{z_3}\right\},\\
&&\hat{\min}\left\{1_{[0,1]}^\infty, l_{z_1}\right\}
\le T(\varphi) \le
\max      \left\{1_{[0,1]}^\infty, l_{z_3}\right\},
\end{eqnarray*}
which by \eqref{eq-more-bounds-on-cvxz} implies that
$\frac{\mu_2(\varphi)}{\mu_2(T(\varphi))} \le \frac{C(n,z_1)}{c(n,z_3)}$.
Since $\nu_2(T(\varphi)) \le \delta_n \mu_2(T(\varphi))$, we get
\begin{eqnarray*}
(\delta_n-\lambda_n)\mu_2(T(\varphi)) &\ge&
\nu_2 (T(\varphi)) - \lambda_n \mu_2 (T(\varphi)) =
\Delta(T(\varphi)) \ge \Delta(\varphi) = \\ &=&
\nu_2(\varphi) - \lambda_n(\varphi) >
- \epsilon \mu_2(\varphi) \ge
-\epsilon \frac{C(n,z_1)}{c(n,z_3)} \mu_2 (T(\varphi)).
\end{eqnarray*}
The latter implies that
$(\lambda_n - \delta_n) <
\epsilon \frac{C(n,z_1)}{c(n,z_3)} =
\frac{1}{2}(\lambda_n-\delta_n)$, which is a contradiction.
\end{proof}
Lemma \ref{lem-same-sups} implies that the maximal value of $s^\J$ on
$\cvx_0(\R^n)$ can be found by studying a function of two variables
$F: [0,\infty) \times [0,\infty] \to \R^+$ given by
\[
F(a,b) =
s^{\J}_n(T_{a, b, 1}) =
\frac
{
\int_0^a e^{-\frac1z}z^{-(n+2)}\left(\frac{z}{a}\right)^n dz +
\int_a^\infty e^{-\frac1z}z^{-(n+2)}\left(\frac{z+b}{a+b}\right)^n dz
}
{
\int_0^a e^{-z}\left(\frac{z}{a}\right)^ndz +
\int_a^\infty e^{-z} \left(\frac{z+b}{a+b}\right)^n dz
},
\]
which is understood for $a=0$ and $b=\infty$ by taking a limit. Note
that $F$ is a rational function of $b$ with coef{}ficients that are
smooth in $a$, as as such is continuous on
$[z_1, z_2] \times [0, \infty]$.
It is easy to verify that when $a\neq0$ and $b < \infty$ we may write
\begin{equation}\label{eq-Fab-as-MN}
F(a,b) =
\frac
{
	(a+b)^n e^{-\frac1a} +
	a^n \int_a^\infty e^{-\frac1z}z^{-(n+2)}(z+b)^n dz
}
{
	(a+b)^n \int_0^a		e^{-z}  z   ^n dz +
	 a	 ^n \int_a^\infty	e^{-z} (z+b)^n dz
}
\end{equation}
Note that by Remark \ref{rem-a-range}, it suffices to look for a
maximum of $F$ when $a \in [z_1, z_2]$.

%

\begin{lem}\label{lem-Maximizer-exists}
There exist $a\in [z_1, z_2],\, b\in[0, \infty]$ such that
\[
s^{\J}_n(T_{a,b,1}) = \lambda_n.
\]
\end{lem}
\begin{proof}
Let $\{a_k\} \subset [z_1, z_2]$, $\{b_k\}\subset [0,\infty]$
be two sequences with $s^{\J}_n(T_{a_k, b_k, 1}) \nearrow \lambda_n$.
Since $a_k \in [z_0, z_1]$, there exists a subsequence $\{a_{k_l}\}$
such that $a_{k_l} \to a$ for some $a\in [z_1, z_2]$. In addition,
there exists a subsequence $\{b_{k_{l_m}}\}$ of $\{b_{k_l}\}$ that
convergences to some $b \in [0,\infty]$. Continuity of $F$ implies
that $s^{\J}_n(T_{a, b, 1}) = \lambda_n$.
\end{proof}
We have seen that $s^{\J}_n$ has a maximizer of the form
$T_{a, b, 1}$. The next two lemmas provide bounds on the parameters
$a,\, b$ of such a maximizer.
\begin{lem} \label{lem-A-Bound}
Let $\alpha\in \left(\frac{1}{2}, 1\right)$. There exists $n_0$
such that if $n > n_0$ and  $T_{a, b, 1}$ is a maximizer of
$s^{\J}_n$, then:
\[
\frac{1}{n+n^{\alpha}} \le a \le \frac{1}{n-n^\alpha}
\]
\end{lem}
\begin{proof}
By remark \ref{rem-a-range}, $a \in [z_1, z_2]$, thus we need to
estimate $z_1, z_2$, the first two roots of $m(z)$. We do this by
considering the following family of functions.
\[
m_\lambda(z) =
e^{-\frac{1}{z}}z^{-(n+2)} - \lambda e^{-z}.
\]
For any $\lambda \in \left[\frac{1}{n!}, \lambda_n\right]$ we may
repeat the proof of Lemma \ref{lem-3-sign-changes}, and deduce that
$\,m_{\lambda}$ has three sign changes, denoted by
$z_1(\lambda) < z_2(\lambda) < z_3(\lambda)$. Since  $m_\lambda(z)$
is decreasing in $\lambda$, we get by \eqref{eq-MJ-prelim-bounds}
that
\[
m(z) = m_{\lambda_n}(z) \le m_{n!}(z),
\]
thus $a \in
[z_1,							   z_2] =
\left[z_1(\lambda_n), z_2(\lambda_n)\right] \subseteq
\left[z_1(n!),				 z_2(n!)\right]$. Note that
$m_{n!}\left(\frac{1}{n}\right)>0$. To check the sign of $m_{n!}$ at
the points $\left(\frac{1}{n\pm n^\alpha}\right)$ we note that:
\[
\frac{\left(1\pm\frac{1}{n^{1-\alpha}}\right)^n}{e^{\pm n^\alpha}}
=
e^{n
\left[
\log\left(1\pm\frac{1}{n^{1-\alpha}}\right)\mp\frac{1}{n^{1-\alpha}}
\right]}
=
e^{-\frac12
\left[
n^{2\alpha - 1} + \overline{o} \left(n^{2\alpha - 1}\right)
\right]
},
\]
which tends to $0$ for $\alpha \in \left(\frac12, 1\right)$. Thus,
\[
m_{n!}\left(\frac{1}{n\pm n^\alpha}\right) =
\left(\frac{n}{e}\right)^n
\left[
(n\pm n^\alpha)^2 
e^{
-\frac12
\left[
	n^{2\alpha - 1} + \overline{o} \left(n^{2\alpha - 1}\right)
\right]
}
- \sqrt{2\pi n}\left(1 + \overline{o} \left( 1 \right)\right)
\right]
< 0
\]
for $n$ large enough. Since
$m_{n!}\left(\frac{1}{n\pm n^\alpha}\right)$ are negative and
$m_{n!}\left(\frac1n\right)$ is positive, we get 
\[
a \in
\left[z_1(n!), z_2(n!)\right] \subset
\left[\frac{1}{n+n^{1-\alpha}}, \frac{1}{n-n^{1-\alpha}}\right].
\]
\end{proof}

\begin{lem}\label{lem-infinite-h}
The function $F(a,b) = s^{\J}_n(T_{a, b, 1})$ is maximized only on
points of the form $(a, \infty)$.
\end{lem}
\begin{proof}
We know by Lemma \ref{lem-Maximizer-exists} that $F$ attains a
maximum, and we will show that it does not attain a maximum at any
point $(a,b)$, where $b \in [0,\infty)$. We shall do this by showing
that $F_b(a,b) > 0$ for $b \neq \infty$, provided that the following
conditions hold:
\begin{equation}\label{eq-ab-is-max}
F(a,b) = \lambda_n,
\end{equation}
\begin{equation}\label{eq-Fa-vanishes}
F_a(a,b) = \frac{\partial F}{\partial a}(a,b) = 0.
\end{equation}
This would imply that $F$ has no (global) maximum points in
$\left[z_1, z_2\right] \times [0,\infty)$.
Recall that by \eqref{eq-Fab-as-MN} we may write,
for $a\neq 0,\,b \neq \infty$:
\[
F(a, b) = \frac{(a+b)^n N_1+a^n N_2}{(a+b)^n M_1 + a^n M_2},
\]
where
\begin{eqnarray*}
	N_1 &=& e^{-\frac1a}, \qquad\qquad\qquad\qquad\qquad\quad\,\,\,
	N_2  =  \int_a^\infty (z+b)^n e^{-\frac1z}z^{-(n+2)} dz, \\
	M_1 &=& \int_0^a z^n e^{-z}dz = \gamma(n+1,a), \qquad
	M_2  =  \int_a^\infty (z+b)^n e^{-z}dz .
\end{eqnarray*}
Computing partial derivatives with respect to $a$ and $b$ yields:
\begin{eqnarray*}
	\frac{\partial N_1}{\partial a} &=& \frac{1}{a^2}e^{-\frac1a}, \qquad
	\frac{\partial N_2}{\partial a}  =  -(a+b)^ne^{-\frac1a}a^{-(n+2)}, \\
	\frac{\partial M_1}{\partial a} &=& a^ne^{-a}, \qquad\,
	\frac{\partial M_2}{\partial a}  =  -(a+b)^n e^{-a}.
\end{eqnarray*}
and
\begin{eqnarray*}
	\frac{\partial N_1}{\partial b} &=& 0, \qquad
	\frac{\partial N_2}{\partial b}  =  n\int_a^\infty (z+b)^{n-1} e^{-\frac1z}z^{-(n+2)} dz \\
	\frac{\partial M_1}{\partial b} &=& 0, \qquad
	\frac{\partial M_2}{\partial b} = n\int_a^\infty (z+b)^{n-1} e^{-z} dz
\end{eqnarray*}
Looking for a possible critical point, we assume \eqref{eq-Fa-vanishes}
holds, and use the relation
$(a+b)^n \frac{\partial N_1}{\partial a} + a^n \frac{\partial N_2}{\partial a} = 0,\,$
and similarly 
$(a+b)^n \frac{\partial M_1}{\partial a} + a^n \frac{\partial M_2}{\partial a} = 0,\,$
to get:
\begin{eqnarray*}
	& &\left[(a+b)^{n-1}N_1 + a^{n-1}N_2\right]\left[(a+b)^n M_1 + a^n M_2\right] =\\
	&=&\left[(a+b)^{n-1}M_1 + a^{n-1}M_2\right]\left[(a+b)^n N_1 + a^n N_2\right],
\end{eqnarray*}
or equivalently $N_1 M_2 = M_1 N_2$, which together with
\eqref{eq-ab-is-max} implies that
\begin{equation}\label{eq-Ni-are-lambda-Mi}
\frac{N_1}{M_1} = \frac{N_2}{M_2} = \lambda_n.
\end{equation}
Next we wish to show that $\frac{\partial F}{\partial b} > 0$, so we
may use \eqref{eq-Ni-are-lambda-Mi} to write:
\begin{eqnarray*}
	( (a+b)^n M_1 &+& a^n M_2 )^2 \frac{\partial F}{\partial b} = \\
	&=&
	\left[ n(a+b)^{n-1}N_1 + na^n \int_a^\infty (z+b)^{n-1} e^{-\frac1z}z^{-(n+2)} dz\right]
	\left[(a+b)^nM_1 + a^n M_2\right] \\ \\
	&-&
	\left[ n(a+b)^{n-1} M_1 + na^n \int_a^\infty (z+b)^{n-1} dz\right]
	\left[ (a+b)^n      N_1 +  a^n N_2\right] \\ \\
	&=& na^n \left[(a+b)^nM_1 + a^n M_2\right]
	\left[
	\int_a^\infty
	(z+b)^{n-1}
	\left( e^{-\frac1z}z^{-(n+2)} - \lambda_n e^{-z} \right) dz
	\right].
\end{eqnarray*}
Recall that $m(z) = e^{-\frac1z}z^{-(n+2)} - \lambda_n e^{-z}$, thus
\begin{equation}\label{eq-sign-of-Fb}
\frac{\partial F}{\partial b} > 0
\qquad\Leftrightarrow\qquad
\int_a^\infty
(z+b)^{n-1} m(z) dz
> 0
\end{equation}
By \eqref{eq-Ni-are-lambda-Mi} we have
$\int_a^\infty (z+b)^n m(z) dz = N_2 - \lambda_n M_2  = 0$, which we
may write as
\begin{eqnarray*}
	P_{n-1}(b) :=
	\int_a^\infty z(z+b)^{n-1} m(z) dz =
	-b\int_a^\infty  (z+b)^{n-1} m(z) dz.
\end{eqnarray*}
By \eqref{eq-sign-of-Fb}, it suffices to show that $P_{n-1}(b) < 0$.
Note that $P_{n-1} (b) = \sum_{i=0}^{n-1} c_k b^k$ is a polynomial of
degree $n-1$ with coef{}ficients
\[
c_k = {{n-1} \choose k}\int_a^\infty z^{n-k} m(z) dz, \qquad k=0,1,\dots n-1.
\]
It is not hard to check that 
\begin{eqnarray*}
	{n-1 \choose k}^{-1}c_k &=&
	\int_a^\infty z^{n-k}e^{-\frac1z}z^{-(n+2)}dz -
	\lambda_n \int_a^\infty z^{n-k}e^{-z} dz =\\ \\
	&=&
	\int_0^{1/a}t^k e^{-t}dt -
	\lambda_n \int_a^\infty z^{n-k}e^{-z} dz = \\ \\
	&=&
	\gamma\left(k+1, \frac{1}{a}\right) - \lambda_n \Gamma(n-k+1, a).
\end{eqnarray*}
Since $k \le n-1$ we have $a\le 1\le n-k$ and
\[
\Gamma(n-k+1,a)\ge
\Gamma(n-k+1,n-k)\ge
\frac{1}{2}(n-k)!\ge
\frac{1}{2}.
\]
For the second inequality $\Gamma(m+1, m) \ge \frac{m!}{2}$ see e.g.
\cite[Equation 8.10.13]{LIES}. Thus, for $n \ge 2$ we have, for every
$k\in\left\{0,\dots,n-1\right\}$
\[
\gamma\left(k+1, \frac{1}{a}\right) - \lambda_n \Gamma(n-k+1,a) <
k! - \frac{1}{2}n! \le
(n-1)! - \frac{1}{2}n! \le 0,
\]
which means $c_k < 0$. We conclude that $P_{n-1}(b) < 0$ for all
$b > 0$, which implies that
\[
\int_a^\infty (z+h)^{n-1} m(z) dz > 0,
\]
i.e. the derivative $\frac{\partial F}{\partial b}$ is positive at
any point $(a, b)$ satisfying \eqref{eq-ab-is-max} and
\eqref{eq-Fa-vanishes}.
\end{proof}
We shall now prove the main theorem.\\ \\
\noindent{\bf Proof of Theorem \ref{Thm-maximizers}.}
By \eqref{Eq-2D-sup}, existence of a maximizer for $s^\J$ is
equivalent to existence of a maximizer for $s_n^\J$, which is
verified in Lemma \ref{lem-Maximizer-exists}. Thus the statement of
Theorem \ref{Thm-maximizers} will follow, if we show that a maximizer
must be of the form $\psi_{K, a}$. Let $\varphi \in \cvx_0(\R^n)$ be
a maximizer of $s^\J$. As before, let $\psi = \varphi^\ast|_{\R^+}$,
where $\varphi^\ast$ is the symmetric of rearrangement $\varphi$.
Then $\psi$ is a maximizer of $s_n^{\J}$, and by Lemma
\ref{lem-Maximizer-unique}, $\psi = T_{\frac{a}{x_0}, b, x_0}$, for some
$a\in (0, \infty),\, b\in[0,\infty],\, x_0\in (0,\infty)$. Note that
$\psi_{x_0} = T_{a, b x_0, 1}$, which by Lemma
\ref{lem-infinite-h} implies $b = \infty$. Thus
$\psi = T_{\frac{a}{x_0}, \infty, x_0}$ and $\varphi^{\ast} =
\max
\left\{
1_{x_0 B_2^n}^\infty \, , \, \frac{a}{x_0}||\cdot||_{B_2^n}
\right\}$.
The level sets of $\varphi^\ast$ are given by
\[
L_z(\varphi^\ast) =
\left\{
\begin{array}{lr}
\frac{x_0}{a} z B_2^n		&~~:~0 \le z\le a \\
x_0 B_2^n				&~:~~~~~~a \le z      \\
\end{array}
\right\}.
\]
This implies that
$z\mapsto \left(\vol_n \left(L_z\varphi\right)\right)^\frac1n$ is
linear on $[0, a]$ and constant on $[a, \infty)$. By the
equality condition of the Brunn-Minkowski inequality applied to the
sets $K := L_{a}(\varphi)$ and $L_{z}(\varphi)$, all level sets of
$\varphi$ are homothetic to $K$. We get
\[
L_z(\varphi) =
\left\{
\begin{array}{lr}
\frac{z}{a} K	&~~:~0 \le z\le a\\
K					&~:~~~~~~a \le z       \\
\end{array}
\right\}.
\]
In other words $\varphi = \psi_{K, a}$ for some positive number
$a$, as required. Note that since $\psi_{x_0} = T_{a, \infty, 1}$
maximizes $s^\J_n$, we get from Lemma \ref{lem-A-Bound} that
\[
\frac{n}{n+n^{\frac23}} \le n a \le \frac{n}{n-n^{\frac23}},
\]
and the proof is complete.
\qed

%
%
%
%
%
%
%
%
%
%

\section{Asymptotically sharp bounds for $s^\J$}
In this section we prove Theorem \ref{Thm-asymp}. In our estimates we
use the Gamma function and the incomplete Gamma functions, defined as
follows. The Gamma function is given by:
\[
\Gamma(n+1) = \int_0^\infty t^{n} e^{-t} dt.
\]
The incomplete Gamma functions are defined by:
\begin{eqnarray*}
	\gamma(n+1, a) &=& \int_0^a t^n e^{-t}dt, \\ \\
	\Gamma(n+1, a) &=& \int_a^\infty t^n e^{-t} dt.
\end{eqnarray*}
Replacing the exponent $e^{-t}$ by its maximal and minimal values
on $[0,a]$ yields, for any $a\le 1$:
\begin{equation} \label{eq-gamma-small-a-bound}
e^{-a}\frac{a^{n+1}}{n+1} \leq \gamma(n+1,a) \le \frac{a^{n+1}}{n+1}.
\end{equation}
The following lemma is an upper bound on $\Gamma(n+1, a)$.
\begin{lem}\label{lem-Gamma-tail}
Let $n\ge1$ and $0 < t < 2(n+1)$. Then
\begin{equation}\label{eq-Gamma-bound}
\int_0^{n+1+t} s^n e^{-s} ds \ge
\left( 1 - e^{-\frac{t^2}{8(n+1)}} \right) n!
\end{equation}
\end{lem}
\begin{proof}
Let $X$ be a random variable with density $\frac{s^n e^{-s}}{n!}$, so
that $\E[X] = n+1$.
For $\theta\in(0,\frac12)$, let $Y = e^{\theta\left(X-\E[X]\right)}$, so
that $\E[Y] = \left( \frac1{e^\theta(1-\theta)}\right)^{n+1} \le
e^{2\theta^2(n+1)}$.
Choose $a = e^{\theta t}$ and use Markov's inequality for $Y$ to get
\begin{eqnarray*}\label{eq-Markov}
P\left(X       \ge n + 1 + t\right) &=&
P\left(X-\E[X] \ge t\right) =
P\left(   Y    \ge a\right) \le
\frac{\E[Y]}{a} \le
e^{2\theta^2(n+1)-\theta t}
\end{eqnarray*}
Optimizing for $\theta$, we choose $\theta = \frac{t}{4(n+1)} < \frac12$
to get
\[
\int_{n+1+t}^\infty \frac{s^n e^{-s} ds}{n!} =
P\left(X       \ge n + 1 + t\right) \le
e^{-\frac{t^2}{8(n+1)}},
\]
which completes the proof.
\end{proof}

\begin{rem}
Lemma \ref{lem-Gamma-tail} is a special case of a more general
concentration property due to Klartag, where $e^{-s}$ is replaced by
an arbitrary log concave function (see Lemmas 4.3 and 4.5 in
\cite{Klartag}).
\end{rem}

By Lemma \ref{lem-infinite-h} we get that, defining $G:(0,\infty)\to
(0,\infty)$ by
\[
G(a) =
s^{\J}_n(T_{a, \infty, 1}) =
\frac
{e^{-\frac1a} + a^n \int_a^\infty e^{-\frac1z} z^{-(n+2)} dz}
{\gamma(n+1,a) + a^ne^{-a}},
\]
we have $\lambda_n = \max_{a\in\left[z_1, z_2\right]} G(a)$.
The curious reader may verify, similarly to the way
\eqref{eq-Ni-are-lambda-Mi} is obtained in the proof of Lemma
\ref{lem-infinite-h}, that $G'(a) = 0$ and $G(a) = \lambda_n$ imply
\begin{equation}\label{eq-N1=lambda-N1}
\lambda_n =
\frac{e^{-\frac1a}}{\gamma(n+1,a)},
\end{equation}
\begin{equation}\label{eq-N2=lambda-N2}
\lambda_n =
\frac{\int_a^\infty e^{-\frac1z} z^{-(n+2)} dz}{e^{-a}} =
e^a \gamma\left(n+1, \frac{1}{a}\right).
\end{equation}
We shall use \eqref{eq-N2=lambda-N2}, together with the bounds on the
partial Gamma functions, to prove the asymptotically sharp bounds on
$\lambda_n$.\\ \\
\noindent{\bf Proof of Theorem \ref{Thm-asymp}.}
We want to show that there exist positive constants $c, C$ such that
for $n$ large enough:
\[
\left(1+\frac{c}{n}\right) n! \le
\lambda_n \le
\left(1+\frac{C}{n}\right) n!.
\]
By Lemma \ref{lem-A-Bound}, for any $\alpha \in (\frac12, 1)$ there
exists $n_0 = n_0(\alpha)$ such that $n > n_0$ implies
$a < \frac{1}{n - n^\alpha}$. Thus, there exists a universal constant
$C > 0$ such that
\[
e^a < 1 + \frac{C}{n}.
\]
Combining the above with \eqref{eq-N2=lambda-N2} yields the required
upper bound:
\[
\lambda_n < \left(1+\frac{C}{n}\right) n!.
\]
For the lower bound, we choose a point $a = \frac{1}{2(n+1)}$ and
show that
$G(a) \ge \left(1+\frac{c}{n}\right) n!$. Note that $\frac1{3n} < a <
\frac1{2n}$. By \eqref{eq-gamma-small-a-bound} we have
\[
\frac{\gamma(n+1,a)}{a^ne^{-a}} \le
\frac{a e^a}{(n+1)} < \frac1{n^2}.
\]
We use \eqref{eq-Gamma-bound}, and the inequalities $e^x > 1 + x$
and $\frac{1}{x + 1} > 1 - x$ ~for positive $x$, to get
\begin{eqnarray*}
G(a) &=&
\frac
{e^{-\frac1a} + a^n \int_a^\infty e^{-\frac1z}z^{-(n+2)} dz}
{\gamma(n+1,a) + a^ne^{-a}} >
\frac{a^n \int_a^\infty e^{-\frac1z}z^{-(n+2)} dz}{\gamma(n+1,a) + a^ne^{-a}} = \\ \\
&=&
\frac{e^a \int_0^\frac1a s^n e^{-s} ds}{\frac{\gamma(n+1,a)}{a^n e^{-a}} + 1} >
(1+a)\left(1 - \frac{\gamma(n+1,a)}{a^n e^{-a}}\right) \int_0^{2(n + 1)} s^n e^{-s} ds
> \\ \\
&>&
\left(1 + \frac1{3n}\right)
\left(1 - \frac1{n^2}\right)
\left(1 - e^{-\frac{n+1}{8}}\right)
n!
\end{eqnarray*}
which for $n$ large enough implies $G(a) > \left(1 + \frac{c}n\right) n!$.

\qed

\newpage
\bibliographystyle{amsplain}

\smallskip \noindent
Dan Florentin \\
Department of Mathematical Sciences  \\
Kent State University, Kent, OH, 44242, USA  \\
{\it email}: danflorentin@gmail.com  \\
%
%

\smallskip \noindent
Alexander Segal \\
Afeka Academic College of Engineering, Tel Aviv, 69107, Israel\\
{\it email}: segalale@gmail.com  \\

\end{document}